\documentclass[10pt,a4paper]{amsart}
\usepackage{esint}
\usepackage{amssymb}
\usepackage{MnSymbol}
\usepackage{comment}

\usepackage[
hypertexnames=false, colorlinks, 
linkcolor=black,citecolor=black,urlcolor=black, linktocpage=true
]%
{hyperref} 
\hypersetup{bookmarksdepth=3}

	\normalbaselines						  %

\usepackage{graphicx}

\newtheorem{thm}{Theorem}[section]

\theoremstyle{definition}
\newtheorem{df}[thm]{Definition}
\newtheorem*{df*}{Definition}

\theoremstyle{remark}

\newtheorem*{rem*}{Remark}

\numberwithin{equation}{section}

\newcommand{\ci}[1]{_{ {}_{\scriptstyle #1}}}

\newcommand{\ti}[1]{_{\scriptstyle \text{\rm #1}}}

\newcommand{\ddoto}{\"{o}}
\newcommand{\ddota}{\"{a}}

\newcommand{\cD}{\mathcal{D}}

\newcommand{\cK}{\mathcal{K}}

\newcommand{\e}{\varepsilon}

\newcommand{\R}{\mathbb{R}}

\newcommand{\ch}{\operatorname{ch}}

\newcommand{\1}{\mathbf{1}}

\newcommand{\cz}{Calder\'{o}n--Zygmund\ }

\newcommand{\La}{\langle }
\newcommand{\Ra}{\rangle }

\newcommand{\fdot}{\,\cdot\,}

\newcommand{\bfD}{\boldsymbol{\mathcal{D}}}

\def\cyr{\fontencoding{OT2}\fontfamily{wncyr}\selectfont}
\DeclareTextFontCommand{\textcyr}{\cyr}

\newcommand{\mathd}{\mathrm{d}}

\newcommand{\tmop}[1]{\ensuremath{\operatorname{#1}}}

\newenvironment{entry}
{\begin{list}{X}%
		{%
			\setlength{\labelwidth}{55pt}%
			\setlength{\leftmargin}{\labelwidth}
			\addtolength{\leftmargin}{\labelsep}%
			\setlength{\itemsep}{.4pc}
	}%
}%
{\end{list}}

\newcounter{vremennyj}

\begin{document}

\title{Dyadic lower little BMO estimates }

\author{K. Domelevo, S. Kakaroumpas, S. Petermichl, O. Soler i Gibert}

\thanks{The authors are partially supported by ERC project CHRiSHarMa no.~DLV-682402 and the Alexander von Humboldt Stiftung.}

\maketitle

\begin{abstract}
  We characterize dyadic little BMO via the boundedness of the tensor
  commutator with a single well chosen dyadic shift. It is shown that several
  proof strategies work for this problem, both in the unweighted case as well as with Bloom weights. Moreover, we address the flexibility of one of our methods.\\
  \textsc{2010 MSC Primary}: 42B35. \textsc{Secondary}: 42B20.
\end{abstract}

\section*{Notation}

\begin{entry}

\item[$\1\ci{E}$] characteristic function of a set $E$;

\item[$\mathd x$] integration with respect to Lebesgue measure; 

\item[$|E|$] $d$-dimensional Lebesgue measure of a measurable set $E\subseteq\R^d$;

\item[$\La f\Ra\ci{E}$] average with respect to Lebesgue measure, $\La f\Ra\ci{E}:=\frac{1}{|E|}\int_{E}f(x)\mathd x$;

\item[$L^{p}(w)$] weighted Lebesgue space, $\|f\|\ci{L^p(w)}^p := \int_{\R^d}|f(x)|^p w(x) \mathd x$; 

\item[$(f,g)$] usual $L^2$-pairing, $(f,g) := \int f(x) \overline{g(x)} \mathd x$;

\item[$w(E)$] Lebesgue integral of a weight $w$ over a set $E$, $w(E):=\int_{E}w(x)\mathd x$;

\item[$p'$] H\"{o}lder conjugate exponent to $p$, $1/p+1/p'=1$; 

\item[$\cD$] family of all dyadic intervals in $\R$;

\item[$\cD(E)$] family of all dyadic intervals $I\in\cD$ contained in a subset $E$ of $\R$;

\item[$\bfD$] family of all dyadic rectangles in $\R^2$;

\item[$\bfD(E)$] family of all dyadic rectangles $R\in\bfD$ contained in a subset $E$ of $\R^2$;

\item[$I_{-},\,I_{+}$] left, right respectively half of an interval $I\in\cD$;

\item[$I'$] sibling in $\cD$ of an interval $I\in\cD$;

\item[$\hat{I}$] parent in $\cD$ of an interval $I\in\cD$;

\item[$h\ci{I}$] $L^{2}$-normalized (cancellative) Haar function for an interval $I\in\cD$, $h\ci{I}:=\frac{\1\ci{I_{+}}-\1\ci{I_{-}}}{\sqrt{|I|}}$;

\item[$f\ci{I}$] Haar coefficient of a function $f\in L^1\ti{loc}(\R)$, $f\ci{I}:=(f,h\ci{I})$, $I\in\cD$;

\item[$h\ci{R}$] $L^2$-normalized (bicancellative) Haar function for a rectangle $R\in\bfD$, $h\ci{R}:= h\ci{I}\otimes h\ci{J}$, where $R = I \times J$;

\item[$f\ci{R}$] usual (biparameter) Haar coefficient of a function $f\in L^1\ti{loc}(\R^2)$, $f\ci{R} :=(f,h\ci{R})$, $R\in\bfD$.

\end{entry}

\section{Introduction}

Let us denote by $\{ h\ci{I} : I \in \mathcal{D} \}$ the Haar base on
$\mathbb{R}$ and let $S$ be the operator densely defined by
\[ h\ci{I_-} \mapsto - h\ci{I_+} \text{ and } h\ci{I_+} \mapsto  h\ci{I_-} . \]
This shift is different from the classical one in \cite{P}. It is time faithful and has other nice properties, in particular it is
an excellent model for the Hilbert transform.

We will work in two parameter space $L^2 (\mathbb{R}^2)$ and we denote by
$S_i$ the shift operator acting in variable $i$. In this note, we mainly
consider the commutator with the tensor product $S\otimes S=S_1S_2$ as follows:
\[ C_b = [S_1 S_2, b] . \]
It is bounded in $L^2$ if and only if the symbol $b$ is in dyadic little BMO.
More precisely,

\begin{thm}
  There holds
  \[ \| b \|_{\tmop{bmo}} \lesssim \| C_b \|\ci{L^2 \rightarrow L^2} \lesssim
     \| b \|_{\tmop{bmo}} \]
  with constants independent of the symbol. The lower norm estimate means that
  \[ \exists C > 0 : \forall b \in \tmop{bmo} \, \exists f \in L^2 : \| C_b f\|\ci{L^2} \geqslant C \| b \|_{\tmop{bmo}} \| f \|\ci{L^2}, \]
  while the upper estimate means, as usual, that
  \[ \exists C : \forall b \in \tmop{bmo} \,\forall f \in L^2 : \| C_b f\|\ci{L^2} \leqslant C \| b \|_{\tmop{bmo}} \| f \|\ci{L^2} . \]
\end{thm}

The theorem holds also with exponents $1<p<\infty$ and in the Bloom setting. Moreover, we show these estimates hold for a certain class of dyadic shifts.

We will provide two proofs of the lower bounds for the commutator with $S_1S_2$ and refer for example to \cite{OP} and \cite{HPW}
for the upper bound. One proof strategy uses an explicit calculation of the
kernel and a modification of the argument of \cite{CRW} and another passes via a
direct testing on a part of the symbol.

This result is not a surprise, it is well known for the Hilbert transform,
where the most elegant argument that is known to us uses Toeplitz operators \cite{FS} and the
elementary characterization of their boundedness to deduce the little BMO
lower estimate from the one parameter result, Nehari's theorem \cite{N}. A real
analytic proof relying on the explicit BMO expression can be found in \cite{OP},
extending a one parameter real variable argument brought forward by \cite{CRW}. As it turns out and as is well known to experts, the argument by
\cite{CRW} relying on the kernel expression can give lower commutator estimates in various settings. It is somewhat surprising, that it can be used for some dyadic operators as well, because the original argument relied heavily on the particular form of the kernels of Hilbert or Riesz transforms and an identity for spherical harmonics. 

Of independent interest is a direct argument by testing on the symbol. It also extends trivially to the rectangular BMO case of the
iterated commutator. Notice that for the rectangular BMO norm we do not have
John--Nirenberg inequalities \cite{BP}, therefore the exponent 2 in the definition
is highly relevant. This case was treated by \cite{FS} for the Hilbert transform
also via direct testing.

Holmes, Treil and Volberg \cite{HTV} have proved a similar result for a different dyadic shift. Blasco and Pott \cite{BP} have an interesting result in the product BMO setting, requiring boundedness of a commutator with a large class of dyadic multiplier operators.

Our proof strategies extend to the weighted Bloom case, in which case we demonstrate lower estimates and commutator charaterizations as well. An interesting result in the Bloom product BMO setting can be found in \cite{KS}.

Our arguments inspired by \cite{CRW} are somewhat flexible and we give certain criteria under which we have a lower bound for a shift operator. These include the one considered in \cite{HTV} and our $S_1S_2$ in this note, showing that both these groups have selected their operators wisely.

\begin{df}
  A function $b (x)$ is in the dyadic BMO space if
  \[ \| b \|\ci{1, \tmop{BMO}} = \sup_I \frac{1}{| I |} \int_I | b (x) -
     \langle b \rangle\ci{I} | \mathd x \]
  is finite. The supremum runs over dyadic intervals.
  
  Since John--Nirenberg holds for this space, we may define the equivalent norm for $1 < p < \infty$
  \[ \| b \|^p\ci{p, \tmop{BMO}} = \sup_I \frac{1}{| I |} \int_I | b (x) -
     \langle b \rangle\ci{I} |^p \mathd x. \]
\end{df}

\begin{df}
  A function $b (x_1, x_2)$ is in the dyadic little BMO space if
  \[ \| b \|\ci{1, \tmop{bmo}} = \sup_R \frac{1}{| R |} \int_R | b (x_1, x_2) -
     \langle b \rangle\ci{R} | \mathd x_1 \mathd x_2 \]
  is finite. The supremum runs over all dyadic rectangles.
  
  Since John--Nirenberg holds for this space, we may define the equivalent norm
  \[ \| b \|^p\ci{p, \tmop{bmo}} = \sup_R \frac{1}{| R |} \int_R | b (x_1, x_2)
     - \langle b \rangle\ci{R} |^p \mathd x_1 \mathd x_2 . \]
\end{df}

\begin{df}
  A function $b (x_1, x_2)$ is in the dyadic rectangular BMO space if
  \[ \| b \|^2\ci{\tmop{BMO}_{\tmop{rec}}} = 
  \sup_R  \frac{1}{| R | } \int_R | b(x_1, x_2) - \langle b \rangle\ci{R_2} (x_1) - \langle b \rangle\ci{R_1}(x_2) + \langle b \rangle\ci{R} |^2 \mathd x_1 \mathd x_2 \]
  is finite. The supremum runs over all dyadic rectangles.
\end{df}

Little BMO can also be realized as a function belonging uniformly to BMO in
each variable separately. The rectangular BMO space can be realized via a
convexity argument as the probabilistic BMO space, where admissible stopping
times are restricted to be tensor products of one-parameter dyadic stopping
times. See \cite{B} for the definition of the probabilistic BMO space in two parameters.

\section{Remarks on one parameter}

\subsection{Lower \texorpdfstring{$\tmop{BMO}$}{BMO} estimate in one parameter via testing}

\

In this section we prove the following lower estimate by testing the
commutator on an appropriate test function:

\begin{thm}
  There holds the one-parameter two-sided estimate
  \[ \| b \|\ci{\tmop{BMO}} \lesssim \| C_b \|\ci{L^2 \rightarrow L^2} \lesssim \| b \|\ci{\tmop{BMO}} . \]
\end{thm}

\begin{proof}
  For any dyadic interval $I$ with parent $\hat{I}$, we will provide a lower
  estimate for
  \[ \| S (b \1\ci{I}) - b S (\1\ci{I}) \|\ci{L^2 (\hat{I})}, \]
  which is bounded above by the full $L^2$ norm of the commutator. It will be
  useful to know how $S$ acts on characteristic functions. There holds
  \begin{eqnarray*}
    S (\1\ci{I}) & = & S \left( \sum_{L} (\1\ci{I}, h\ci{L}) h\ci{L} \right)\\
    & = & S \left( \sum_{L : L \supsetneqq I} (\1\ci{I}, h\ci{L}) h\ci{L} \right) .
  \end{eqnarray*}
  A simple but important observation is that $S (\1\ci{I})$ is supported outside
  of $\hat{I}$. Indeed, for intervals $L \supsetneqq I$ observe that $S (h\ci{L})$ is
  supported on the dyadic sibling $L'$ of $L$ and has therefore no support on
  $I$ or $I'$. Let us now consider the local part $P\ci{I}b$ defined as
	  \[ P\ci{I}b = \sum_{K \in \mathcal{D} (I)} b\ci{K} h\ci{K} . \]
  We calculate
  \[ S (P\ci{I}b \cdot \1\ci{I}) = S \left( \sum_{K \in \mathcal{D} (I)} b\ci{K} h\ci{K}
     \cdot \1\ci{I} \right) = S \left( \sum_{K \in \mathcal{D} (I)} b\ci{K} h\ci{K}
     \right) . \]
  Then we calculate
  \[ P\ci{I}b \cdot S (\1\ci{I}) = \sum_{K \in \mathcal{D} (I)} b\ci{K} h\ci{K} \left[ S
     \left( \sum_{L : L \supsetneqq I} (\1\ci{I}, h\ci{L}) h\ci{L} 
     \right) \right]
     = 0. \]
  Indeed, $S (\1\ci{I})$ has no support in $I$ but $P\ci{I}b$ is only supported on
  $I$. Let us now consider the outer part
  \[ P\ci{I^c}b = \sum_{K \nin \mathcal{D} (I)} b\ci{K} h\ci{K} . \]
  Observe that $P\ci{I^c}b$ is constant on $I$: if $K \cap I = \varnothing$ that
  constant is 0. If $I \subsetneqq K$ then $h\ci{K}$ is constant on $I$. Therefore
  $S (P\ci{I^c}b \cdot \1\ci{I}) = P\ci{I^c}b (I) S (\1\ci{I})$ and this part does not
  have a contribution on $\hat{I}$. Likewise $P\ci{I^c}b S (\1\ci{I})$ has no
  contribution on $\hat{I}$. Gathering the information, we have
  \begin{eqnarray*}
    \| [S, b] \1\ci{I} \|^2\ci{L^2 (\hat{I})} & = & \| [S, P\ci{I}b] \1\ci{I} + [S,
    P\ci{I^c}b] \1\ci{I} \|^2\ci{L^2 (\hat{I})}\\
    & = & \left\| S \left( \sum_{K \in \mathcal{D} (I)} b\ci{K} h\ci{K} \right)
    \right\|^2\ci{L^2 (\hat{I})}\\
    & = & \int_I | b - \langle b \rangle\ci{I} |^2 .
  \end{eqnarray*}
  Here we used the fact that the support of $S (P\ci{I}b)$ is contained in
  $\hat{I}$ and that $S$ is an isometry in $L^2$. These considerations tell us
  that the lower BMO estimate is seen when testing on $\1\ci{I}$
  and taking the supremum in $I$.
  \[ \| b \|\ci{\tmop{BMO}} \leqslant \sup_I \left\| [S, b] \frac{\1\ci{I}}{| I
     |^{1/2}} \right\|\ci{L^2} . \]
  The upper estimate is similar to the one for the classical Haar shift and is
  omitted.
\end{proof}

\section{Two parameter implications}

\subsection{Lower little BMO estimate via testing} \label{littlebmo}

\begin{thm}
  There holds the two-parameter two-sided estimate
  \[ \| b \|\ci{\tmop{bmo}} \lesssim \| C_b \|\ci{L^2 \rightarrow L^2} \lesssim \|
     b \|\ci{\tmop{bmo}} . \]
\end{thm}

\begin{proof}
  Let us write for $R\in\bfD$, $R=R_1\times R_2$, define the domain $\check R= (\hat{R_1}\times \R) \cup (\R \times \hat{R_2})$ and define $\mathcal{R}^c= \{K\in \bfD:R_i \subsetneqq K_i \text{ for } i=1,2 \}$. Then, let
  \[ P\ci{\mathcal{R}}b = \sum_{K \in \mathcal{R} } b\ci{K} h\ci{K} \text{ and } P\ci{\mathcal{R}^c}b = \sum_{K
     \in \mathcal{R}^c } b\ci{K} h\ci{K} . \]
  Observe that the sum over $K \in \mathcal{R}$ can be split as
  \begin{equation*}
      P\ci{\mathcal{R}}b = \sum_{K\in\bfD} b\ci{K} h\ci{K}
      - \sum_{K_1\nin\mathcal{D}(R_1)} \sum_{K_2\nin\mathcal{D}(R_2)} b\ci{K_1\times K_2} h\ci{K_1\times K_2}.
  \end{equation*}
  The first term in the previous expression is just $b(x,y),$ while for $(x,y) \in R,$
  the second term is just $\langle b \rangle\ci{R}$ (it is enough to apply the one-parameter argument on each variable separately).
  In other words, we have that
  \begin{equation}
      \label{eq:LocalisedFunctionB}
      P\ci{\mathcal{R}}b (x,y) = b(x,y) - \langle b \rangle\ci{R}
  \end{equation}
  for all $(x,y) \in R.$
  We test on $\1\ci R$, split the commutator $C_b=C_{P\ci{\mathcal{R}}b} + C_{P\ci{\mathcal{R}^c}b} $ and integrate only in $\check{R}$. Observe that 
  $(P\ci{\mathcal{R}}b+P\ci{\mathcal{R}^c}b) S_1 S_2 (\1\ci{R})$ has no contribution in $\check{R}$. Likewise, observe that when $K\in \mathcal{R}^c$, then $h\ci K$ is constant on $R$ and $S_1 S_2
  (h\ci K\1\ci{R}) = h\ci K(R)S_1S_2 (\1\ci{R})$ with no contribution on $\check{R}$. The remaining term is $S_1S_2(P\ci{\mathcal{R}}b \1\ci R)$. Now, we observe that the support of $S_1S_2(P\ci{\mathcal{R}}b \1\ci R)$ lies in $\check{R}$. To this end, we only need to check the action of the shifts on functions $h\ci{K}\1\ci{R}$ with $K \in \mathcal{R}.$ Indeed, when $K\cap R=\emptyset$, then $h\ci K \1\ci R=0$. The $K \in \mathcal{R}$ that remain have $K_i\subseteq R_i$ for either $i=1$ or $i=2$. Let us assume that $K_1 \subseteq R_1,$ as the other case follows the same argument. It holds that $S_1(h\ci K \1\ci R)$ has support in $\hat{K_1}\times \R$ and so $S_1S_2(h\ci{K} \1\ci R)$ has support in $\check{R}$. Finally, $P\ci{\mathcal{R}}b \1 \ci R=\1\ci R(b-\langle b \rangle\ci R),$ which follows from \eqref{eq:LocalisedFunctionB}.

   Gathering
  the information, we have
  \begin{eqnarray*}
    \| [S_1 S_2, b] \1\ci{R} \|\ci{L^2 (\check{R})}^2 
    & = & \left\| S_1 S_2( P\ci{\mathcal{R}}b \1\ci{R}) \right\|^2\ci{L^2 (\check{R})}\\
    & = & \left\| S_1 S_2( P\ci{\mathcal{R}}b \1\ci{R}) 
    \right\|^2_{L^2}\\
    & = & \int_R | b(x_1,x_2) - \langle b \rangle\ci{R} |^2 \mathd x_1 \mathd x_2 .
  \end{eqnarray*}
  
\end{proof}

\subsection{Lower rectangular BMO estimate for iterated commutator}\label{rect}

\

In this section we prove a lower rectangular BMO estimate by testing the
iterated commutator $C^{\tmop{it}}_b = [S_1, [S_2, b]]=[S_2, [S_1, b]]$ on $\1\ci{R}$.

\begin{thm}
  There holds
  \[ \| b \|\ci{\tmop{BMO}_{\tmop{rec}}} \lesssim \| C^{\tmop{it}}_b \|\ci{L^2
     \rightarrow L^2} . \]
\end{thm}

The proof is similar to the one parameter case above. Let $R = R_1 \times R_2$
be a dyadic rectangle.

We calculate $C^{\tmop{it}}_b \1\ci{R}$. For that, introduce $\varphi = [S_1, b]
(\1\ci{R_1})$ and observe that
\[ [S_1, b] (\1\ci{R}) = [S_1, b] (\1\ci{R_1}) \1\ci{R_2} = \varphi \1\ci{R_2},
\]
and 
\[ [S_1, b] (S_2 (\1\ci{R})) = [S_1, b] (\1\ci{R_1}) S_2 (\1\ci{R_2}) = \varphi
   S_2 (\1\ci{R_2}) . \]
It follows
\[ C^{\tmop{it}}_b \1\ci{R} = S_2 ([S_1, b] (\1\ci{R})) - [S_1, b] (S_2 (\1\ci{R}))
   = S_2 (\varphi \1\ci{R_2}) - \varphi S_2 (\1\ci{R_2}) . \]
Now we integrate the commutator, using the facts learned in the one-parameter
case.
\begin{align*}
  &\int_{\widehat{R_1}} \int_{\widehat{R_2}} | [S_2, \varphi] \1\ci{R_2} |^2 \mathd
  x_2 \mathd x_1 \\
  &=  \int_{\widehat{R_1}} \int_{R_2} | \varphi - \langle
  \varphi \rangle\ci{R_2} |^2 \mathd x_2 \mathd x_1\\
  &=  \int_{\widehat{R_1}} \int_{R_2} | [S_1, b] \1\ci{R_1} - \langle [S_1,
  b] \1_{R_1} \rangle\ci{R_2} |^2 \mathd x_2 \mathd x_1\\
  &=  \int_{\widehat{R_1}} \int_{R_2} | [S_1, b - \langle b \rangle\ci{R_2}]
  \1\ci{R_1} |^2 \mathd x_2 \mathd x_1\\
  &=  \int_{R_2} \int_{R_1} | b - \langle b \rangle\ci{R_2} - \langle b -
  \langle b \rangle\ci{R_2} \rangle\ci{R_1} |^2 \mathd x_1 \mathd x_2\\
  & =  \int_{R_2} \int_{R_1} | b (x_1, x_2) - \langle b \rangle\ci{R_2} (x_1)
  - \langle b \rangle\ci{R_1} (x_2) + \langle b \rangle\ci{R} |^2 \mathd x_1 \mathd
  x_2 .
\end{align*}
Taking the supremum over $R$ delivers the rectangular BMO lower estimate.

\section{Lower estimates via the kernel}

In this section we work in $L^p$ --- in the light of the upper estimates, this is not needed in the unweighted case, but will be of value in the weighted case later on.

\subsection{Kernel representation}
First, we establish the kernel of the operator $S_1S_2$. We recall that in one parameter
\begin{equation*}
Sf:=\sum_{I\in\cD}(f\ci{I_{+}}h\ci{I_{-}}-f\ci{I_{-}}h\ci{I_{+}}).
\end{equation*}
Thus in two parameters (now writing $R=I\times J$ to avoid indices)
\begin{equation*}
(S\otimes S)f=S_1S_2f=\sum_{I,J\in\cD}\sum_{\e,\delta\in\lbrace-,+\rbrace}\e\delta f\ci{I_{\e}\times J_{\delta}}h\ci{I_{-\e}\times J_{-\delta}}.
\end{equation*}
The operator $S_1S_2$ has a formal kernel. Namely
\begin{equation*}
S_1S_2f(x)=\int_{\R^2}\cK(x,y)f(y)\mathd y,
\end{equation*}
where
\begin{equation}
\label{kernel}
\cK(x,y):=\sum_{I,J\in\cD}\sum_{\e,\delta\in\lbrace-,+\rbrace}\e\delta h\ci{I_{\e}\times J_{\delta}}(y)h\ci{I_{-\e}\times J_{-\delta}}(x),
\end{equation}
for all $x=(x_1,x_2),y=(y_1,y_2)\in\R^2$ with $x_i\neq y_i,~i=1,2$, and $\cK(x,y):=0$ if $x_1=y_1$ or $x_2=y_2$. The kernel $\cK$ given by formula \eqref{kernel} is well-defined pointwise. In fact, for each $x,y\in\R^2$ with $x_i\neq y_i,~i=1,2$ there exists at most one dyadic rectangle $I\times J$ such that $\sum_{\e,\delta\in\lbrace-,+\rbrace}\e\delta h\ci{I_{\e}\times J_{\delta}}(y)h\ci{I_{-\e}\times J_{-\delta}}(x)\neq0$, and then exactly one of the products $h\ci{I_{\e}\times J_{\delta}}(y)h\ci{I_{-\e}\times J_{-\delta}}(x),~\e,\delta\in\lbrace-,+\rbrace$ is nonzero ($I\times J$ is then the minimal dyadic rectangle containing both $x$ and $y$). In particular, for all $I,J\in\cD$ and for all $x,y\in I\times J$ with $x_i\neq y_i$, $i=1,2$, we have $\cK(x,y)\neq0$ and
\begin{equation}
\label{kernel over rectangle}
\cK(x,y)=\sum_{\e,\delta\in\lbrace-,+\rbrace}\sum_{\substack{K\in\cD(I)\\L\in\cD(J)}}\e\delta h\ci{K_{\e}\times L_{\delta}}(y)h\ci{K_{-\e}\times L_{-\delta}}(x).
\end{equation}
For all $x\in\R^2$, set
\begin{align*}
A_{x}&:=\lbrace y\in\R^2:~\cK(x,y)\neq0\rbrace\\
&=\lbrace y\in\R^2:~~x_i\neq y_i,~i=1,2,~\exists I,J\in\cD\text{ such that }x,y\in I\times J\rbrace.
\end{align*}
Then the previous observations imply that for all $I,J\in\cD$ and for \emph{all} $x,y\in\R^2$ with $x_i\neq y_i$, $i=1,2$ we have
\begin{equation}
\label{inverse kernel over rectangle}
\1\ci{I\times J}(x)\cdot\1\ci{I\times J}(y)\cdot\frac{\1\ci{A_{x}}(y)}{\cK(x,y)}=\sum_{\substack{K\in\cD(I)\\L\in\cD(J)}}\sum_{\e,\delta\in\lbrace-,+\rbrace}\e\delta\frac{\1\ci{K_{\e}\times L_{\delta}}(y)}{h\ci{K_{\e}\times L_{\delta}}(y)}\cdot\frac{\1\ci{K_{-\e}\times L_{-\delta}}(x)}{h\ci{K_{-\e}\times L_{-\delta}}(x)},
\end{equation}
where we adopt the convention $\frac{0}{0}=0$.

\subsection{A heuristic argument} 

We want to prove that
\begin{equation}
\label{unweighted goal}
\left(\frac{1}{|I\times J|}\int_{I\times J}|b(x)-\La b\Ra\ci{I\times J}|^p\mathd x\right)^{1/p}\lesssim\Vert[S_1S_2,b]\Vert\ci{L^p(\R^2)\rightarrow L^p(\R^2)},\qquad\forall I,J\in\cD.
\end{equation}
Note that formally
\begin{equation}
\label{formal_commutator}
[S_1S_2,b]f(x)=\int_{\R^2}(b(x)-b(y))\cK(x,y)f(y)\mathd y.
\end{equation}
Let $I,J\in\cD$ be arbitrary. Fix $x\in I\times J$. Then
\begin{equation*}
(I\times J)\setminus((\lbrace x_1\rbrace\times\R)\cup(\R\times\lbrace x_2\rbrace))\subseteq A_x.
\end{equation*}
Therefore, we can formally write
\begin{align*}
&|I\times J|\1\ci{I\times J}(x)(b(x)-\La b\Ra\ci{I\times J})=\1\ci{I\times J}(x)\int_{I\times J}(b(x)-b(y))\mathd y\\
&=\1\ci{I\times J}(x)\int_{A_x}(b(x)-b(y))\1\ci{I\times J}(y)\mathd y\\
&=\1\ci{I\times J}(x)\int_{\R^2}(b(x)-b(y))\1\ci{I\times J}(y)\cK(x,y)\frac{\1\ci{A_x}(y)}{\cK(x,y)}\mathd y\\
&=[S_1S_2,b]\left(\1\ci{I\times J}(x)\cdot\1\ci{I\times J}\cdot\frac{\1\ci{A_x}}{\cK(x,\fdot)}\right)(x)\\
&=\sum_{\substack{K\in\cD(I)\\L\in\cD(J)}}\sum_{\e,\delta\in\lbrace-,+\rbrace}\e\delta[S_1S_2,b]\left(\frac{\1\ci{K_{\e}\times L_{\delta}}}{h\ci{K_{\e}\times L_{\delta}}}\right)(x)\cdot\frac{\1\ci{K_{-\e}\times L_{-\delta}}(x)}{h\ci{K_{-\e}\times L_{-\delta}}(x)}.
\end{align*}
Thus, taking the $L^p_{x}(\R^2)$ norms of both sides, writing $K'$ for the sibling of a dyadic interval $K$, and using the triangle inequality we deduce
\begin{align*}
&\Vert |I\times J|\1\ci{I\times J}(x)(b(x)-\La b\Ra\ci{J})\Vert\ci{L^p_x(\R)}\\
&\leq
\Vert [S_1S_2,b]\Vert\ci{L^p(\R^2)\rightarrow L^p(\R^2)}\sum_{\substack{K\in\cD(I)\setminus\lbrace I\rbrace\\L\in\cD(J)\setminus\lbrace J\rbrace}}\left\Vert \frac{\1\ci{K\times L}(x)}{h\ci{K\times L}(x)}\right\Vert\ci{L^{\infty}_{x}(\R^2)}\left\Vert\frac{\1\ci{K'\times L'}}{h\ci{K'\times L'}}\right\Vert\ci{L^p(\R^2)}.
\end{align*}
We have
\begin{equation*}
\left\Vert \frac{\1\ci{K\times L}(x)}{h\ci{K\times L}(x)}\right\Vert\ci{L^{\infty}_{x}(\R^2)}=\sqrt{|K\times L|},\qquad \left\Vert\frac{\1\ci{K'\times L'}}{h\ci{K'\times L'}}\right\Vert\ci{L^p(\R^2)}=|K'\times L'|^{\frac{1}{2}+\frac{1}{p}}=|K\times L|^{\frac{1}{2}+\frac{1}{p}}.
\end{equation*}
Thus
\begin{align*}
&\sum_{\substack{K\in\cD(I)\setminus\lbrace I\rbrace\\L\in\cD(J)\setminus\lbrace J\rbrace}}\left\Vert \frac{\1\ci{K\times L}(x)}{h\ci{K\times L}(x)}\right\Vert\ci{L^{\infty}_{x}(\R^2)}\left\Vert\frac{\1\ci{K_{s}\times L_s}(y)}{h\ci{K_{s}\times L_s}(y)}\right\Vert\ci{L^p_{y}(\R^2)}\\
&=\sum_{\substack{K\in\cD(I)\setminus\lbrace I\rbrace\\L\in\cD(J)\setminus\lbrace J\rbrace}}|K\times L|^{1+\frac{1}{p}}=
\bigg(\sum_{\substack{K\in\cD(I)\\K\neq I}}|K|^{1+\frac{1}{p}}\bigg)\bigg(\sum_{\substack{L\in\cD(J)\\L\neq J}}|L|^{1+\frac{1}{p}}\bigg).
\end{align*}
We have
\begin{align*}
\sum_{\substack{K\in\cD(I)\\K\neq I}}|K|^{1+\frac{1}{p}}=|I|^{1+\frac{1}{p}}\sum_{n=1}^{\infty}2^{n}\cdot2^{-n\left(1+\frac{1}{p}\right)}=c_p|I|^{1+\frac{1}{p}},
\end{align*}
and similarly for $J$. Thus
\begin{equation*}
|I\times J|\left(\int_{I\times J}|b(x)-\La b\Ra\ci{J}|^p\mathd x\right)^{1/p}\leq\Vert [S_1S_2,b]\Vert\ci{L^p(\R^2)\rightarrow L^p(\R^2)}c_{p}^2|I\times J|^{1+\frac{1}{p}},
\end{equation*}
which immediately implies \eqref{unweighted goal} (with constant $c_p^2$).

\subsection{A rigorous argument}

\subsubsection{Truncated kernel} 

For all positive integers $n$, denote by $\cK_{n}$ the truncated kernel given pointwise by
\begin{equation}
\label{truncated kernel}
\cK_n(x,y):=\sum_{\e,\delta\in\lbrace-,+\rbrace}\sum_{\substack{I,J\in\cD\\2^{-n}\leq|I|,|J|\leq2^{n}}}\e\delta h\ci{I_{\e}\times J_{\delta}}(y)h\ci{I_{-\e}\times J_{-\delta}}(x),
\end{equation}
for all $x,y\in\R^2$ with $x_i\neq y_i$, $i=1,2$, and $\cK_n(x,y):=0$ whenever $x_1=y_1$ or $x_2=y_2$.
Note that for every fixed $x\in\R^2$, the sum in \eqref{truncated kernel} is always finite. For all $x\in\R^2$, set
\begin{equation*}
A_{n,x}:=\lbrace y\in\R^2:~\cK_{n}(x,y)\neq0\rbrace.
\end{equation*}
Clearly $A_{n,x}\subseteq A_{n+1,x}$ and $\bigcup_{n=0}^{\infty}A_{n,x}=A_x$. Note that $\cK(x,y)=\cK_{n}(x,y)$, for all $y\in A_{n,x}$ and for all $x\in\R^2$, since for each fixed $x,y\in\R^2$ there exists at most one $I\times J$ having nonzero contribution in the sum in \eqref{kernel}. It follows by \eqref{inverse kernel over rectangle} that for all $I,J\in\cD$ and for \emph{all} $x,y\in\R^2$ with $x_i\neq y_i$, $i=1,2$ we have
\begin{equation}
\label{inverse truncated kernel}
\1\ci{I \times J}(x)\cdot\1\ci{I \times J}(y)\cdot\frac{\1\ci{A_{n,x}}(y)}{\cK_{n}(x,y)}=\sum_{\substack{K\in\cD(I)\\L\in\cD(J)\\2^{-n}\leq|K|,|L|\leq2^{n}}}\sum_{\e,\delta\in\lbrace-,+\rbrace}\e\delta\frac{\1\ci{K_{\e}\times L_{\delta}}(y)}{h\ci{K_{\e}\times L_{\delta}}(y)}\cdot\frac{\1\ci{K_{-\e}\times L_{-\delta}}(x)}{h\ci{K_{-\e}\times L_{-\delta}}(x)}.
\end{equation}

\subsubsection{Truncated operators}

Fix a positive integer $n$. Clearly, the operator $T_n$ given pointwise \emph{everywhere} by
\begin{equation*}
T_nf(x):=\int_{\R^2}\cK_n(x,y)f(y)\mathd y=\sum_{\e,\delta\in\lbrace-,+\rbrace}\sum_{\substack{I,J\in\cD\\2^{-n}\leq|I|,|J|\leq2^{n}}}\e\delta f\ci{I_{\e}\times J_{\delta}}h\ci{I_{-\e}\times J_{-\delta}}(x),
\end{equation*}
for $x\in \R^2$ and $f\in L^1\ti{loc}(\R^2)$, is well-defined. In fact, for every fixed $x\in\R^2$, the sum in the last line above is finite.

\subsubsection{Computations} 

Let now $I,J\in\cD$ be arbitrary. Notice that for all $x\in I\times J$, we have $(I\times J)\setminus ((\lbrace x_1\rbrace\times\R)\cup(\R\times\lbrace x_2\rbrace))\subseteq A_x$. Therefore, for \emph{all} $x\in\R^2$ we have
\begin{align*}
&|I\times J|\1\ci{I\times J}(x)(b(x)-\La b\Ra\ci{I\times J})=\1\ci{I\times J}(x)\int_{I\times J}(b(x)-b(y))\mathd y\\
&=\1\ci{I\times J}(x)\int_{A_x}(b(x)-b(y))\1\ci{I\times J}(y)\mathd y\\
&=\lim_{n\rightarrow\infty}\1\ci{I\times J}(x)\int_{A_{n,x}}(b(x)-b(y))\1\ci{I\times J}(y)\mathd y\\
&=\lim_{n\rightarrow\infty}\1\ci{I\times J}(x)\int_{\R^2}(b(x)-b(y))\1\ci{I\times J}(y)\cK_{n}(x,y)\frac{\1\ci{A_{n,x}}(y)}{\cK_{n}(x,y)}\mathd y\\
&=\lim_{n\rightarrow\infty}[T_n,b]\left(\1\ci{I\times J}(x)\cdot\1\ci{I\times J}\cdot\frac{\1\ci{A_{n,x}}}{\cK_{n}(x,\fdot)}\right)(x)\\
&=\lim_{n\rightarrow\infty}\bigg[\sum_{\substack{K\in\cD(I)\\L\in\cD(J)\\2^{-n}\leq|K|,|L|\leq 2^{n}}}\sum_{\e,\delta\in\lbrace-,+\rbrace}\e\delta[T_n,b]\left(\frac{\1\ci{K_{\e}\times L_{\delta}}}{h\ci{K_{\e}\times L_{\delta}}}\right)(x)\cdot\frac{\1\ci{K_{-\e}\times L_{-\delta}}(x)}{h\ci{K_{-\e}\times L_{-\delta}}(x)}\bigg].
\end{align*}
Note that all sums in the last line are finite, so we have no convergence issues. Notice now that for all $K,L\in\cD$ with $2^{-n}\leq|K|,|L|\leq 2^{n}$ and for any $\e,\delta\in\lbrace+,-\rbrace$ we have
\begin{equation*}
\1\ci{K_{-\e}\times L_{-\delta}}T_n(f\1\ci{K_{\e}\times L_{\delta}})=\e\delta f\ci{K_{\e}\times L_{\delta}}h\ci{K_{-\e}\times L_{-\delta}}=\1\ci{K_{-\e}\times L_{-\delta}}S_1S_2(f\1\ci{K_{\e}\times L_{\delta}})\qquad\text{a.e.},
\end{equation*}
which clearly implies
\begin{equation*}
\1\ci{K_{-\e}\times L_{-\delta}}[T_n,b](f\1\ci{K_{\e}\times L_{\delta}})=\1\ci{K_{-\e}\times L_{-\delta}}[S_1S_2,b](f\1\ci{K_{\e}\times L_{\delta}})\qquad\text{a.e.}
\end{equation*}
Since $\cD$ is a \emph{countable} set, it follows that for all positive integers $n$, there exists a Borel set $E_n$ in $\R^2$ of zero measure, such that for all $K,L\in\cD$, for any $\e,\delta\in\lbrace+,-\rbrace$ and for any $x\in\R^2\setminus E_n$, there holds
\begin{equation*}
\frac{\1\ci{K_{-\e}\times L_{-\delta}}(x)}{h\ci{K_{-\e}\times L_{-\delta}}(x)}[T_n,b]\left(\frac{\1\ci{K_{\e}\times L_{\delta}}}{h\ci{K_{\e}\times L_{\delta}}}\right)(x)=\frac{\1\ci{K_{-\e}\times L_{-\delta}}(x)}{h\ci{K_{-\e}\times L_{-\delta}}(x)}[S_1S_2,b]\left(\frac{\1\ci{K_{\e}\times L_{\delta}}}{h\ci{K_{\e}\times L_{\delta}}}\right)(x).
\end{equation*}
Set $E:=\bigcup_{n=1}^{\infty}E_n$. Then, $E$ is a set of zero measure in $\R$, and we have
\begin{align*}
&|I\times J|\1\ci{I\times J}(x)(b(x)-\La b\Ra\ci{I\times J})\\
&=\lim_{n\rightarrow\infty}\bigg[\sum_{\substack{K\in\cD(I)\\L\in\cD(J)\\2^{-n}\leq|K|,|L|\leq 2^{n}}}\sum_{\e,\delta\in\lbrace-,+\rbrace}\e\delta[S_1S_2,b]\left(\frac{\1\ci{K_{\e}\times L_{\delta}}}{h\ci{K_{\e}\times L_{\delta}}}\right)(x)\cdot\frac{\1\ci{K_{-\e}\times L_{-\delta}}(x)}{h\ci{K_{-\e}\times L_{-\delta}}(x)}\bigg],
\end{align*}
for all $x\in\R^2\setminus E$, and therefore for \emph{almost every} $x\in\R^2$. Hence, taking the $L^p_{x}(\R^2)$ norm of both sides in the last display, using Fatou's lemma (to pass the limit outside the $L^p_{x}(\R^2)$ norm) and then working exactly as in the heuristic argument above, we conclude again \eqref{unweighted goal} (with the same constant $c_p^2$ as in the heuristic argument).

\section{Weighted estimates}

We extend the unweighted bound to a weighted bound at little extra effort. Originally, this type of weighted inequality was introduced by \cite{Bl} and has found new interest, modern proofs and notable generalizations in \cite{HLW}, \cite{HLW2}. Meanwhile, a number of articles are written on the subject of weighted commutators, but most concern upper estimates rather than charaterizations. 

Here and in what follows, recall that a weight $w$ on $\R$ is a locally integrable function on $\R$ which is positive almost everywhere.
We say that a weight $w$ on $\R$ is an $A_p$ weight, for $1 < p < \infty,$ if
\begin{equation}
    \label{eq:ApCondition}
    [w]\ci{A_p} \colon = \sup_{I} \left(\frac{1}{|I|} \int_I w(x)\mathd x\right) \left(\frac{1}{|I|} \int_I (w(x))^{-1/(p-1)}\mathd x\right)^{p-1} < \infty,
\end{equation}
where the supremum ranges over all finite intervals $I \subseteq \R.$
Similarly, we say that a weight $w$ is a dyadic $A_p$ weight, for $1 < p < \infty,$ if the supremum in condition \eqref{eq:ApCondition} ranges only over dyadic intervals $I \in \cD.$
These definitions also extend to the bi-parameter case.
Namely, a weight $w$ on $\R^2$ is a locally integrable function on $\R^2$ which is positive almost everywhere.
Also, we say that a weight $w$ on $\R^2$ is a bi-parameter $A_p$ weight, for $1 < p < \infty,$ if
\begin{equation}
    \label{eq:BiparameterApCondition}
    [w]\ci{A_p} \colon = \sup_{R} \left(\frac{1}{|R|} \int_R w(x)\mathd x\right) \left(\frac{1}{|R|} \int_R (w(x))^{-1/(p-1)}\mathd x\right)^{p-1} < \infty,
\end{equation}
where the supremum ranges over all finite rectangles $R \subseteq \R^2$ of positive measure with sides parallel to the axes.
As before, we say that a weight $w$ is a dyadic bi-parameter $A_p$ weight, for $1 < p < \infty,$ if the supremum in condition \eqref{eq:BiparameterApCondition} ranges only over dyadic rectangles $R.$

\begin{thm}
  \label{thm:LowerBoundByTesting}
  There holds for dyadic bi-parameter $A_p$ weights $\mu,\lambda$, $1<p<\infty$
  \[ \| b \|\ci{p,\tmop{bmo} (\mu,\lambda)} \lesssim \| C_b \|\ci{L^p ({\mu})
     \rightarrow L^p (\lambda)} \lesssim \| b \|\ci{p,\tmop{bmo} (\mu,\lambda)} \]
  with constants independent of the symbol. 
\end{thm}

Recall that the weighted little $\tmop{BMO}$ space can be equivalently defined, as in \cite{HPW}, by the norm
\[\| b \|\ci{p,\tmop{bmo} (\mu,\lambda)}= \sup_R \left( \frac{1}{{\mu} (R)} \int_R | b (x) - \langle b \rangle\ci{R}
   |^p \lambda (x) \mathd x \right)^{1 / p} \]
for $1 < p < \infty.$
   
   See also \cite{HPW} for the upper estimates.

\subsection{Weighted estimate by testing}

Most lower norm estimates in the Bloom setting have been obtained via a use of the argument in  \cite{CRW}, see for example \cite{HLW}, \cite{HLW2} and \cite{HPW}. In \cite{KS} a different argument was needed to give lower norm estimates in a product setting. 

In this section we demonstrate that the lower bounds can be obtained by testing. We work with the exponent $p=2$, the argument works for other $p$. We show the argument in one parameter and point out the modifications needed to pass to the little BMO case in $L^p$ and the rectangular BMO case in $L^2$ for the iterated commutator.

\begin{proof}[Proof of Theorem~\ref{thm:LowerBoundByTesting}]
  As before, for any dyadic interval $I$ with parent $\hat{I}$, we will provide a lower
  estimate for
  \[
\| S (b \1\ci{I}) - b S (\1\ci{I}) \|\ci{L^2 (\hat{I}, \lambda)}, \]
  which is bounded above by $\Vert[S,b]\Vert\ci{L^2(\mu) \rightarrow L^2(\lambda)} \Vert\1\ci{I}\Vert\ci{L^2(\mu)}.$  We have
  \[ \| [S, b] \1\ci{I} \|\ci{L^2 (\hat{I}, \lambda)}^2 = \| [S, P\ci{I}b] \1\ci{I}
     + [S, P\ci{I^c}b] \1\ci{I} \|^2\ci{L^2 (\hat{I}, \lambda)} = \left\| S \left(
     \sum_{K \in \mathcal{D} (I)} b\ci{K} h\ci{K} \right) \right\|^2\ci{L^2 (\hat{I},
     \lambda)} . \]
 Using the $A_2$ characteristic of $\lambda$ and the fact that $S
  : L^2 (\lambda) \rightarrow L^2 (\lambda)$ is bounded, we obtain also a
  lower bound $\| f \|\ci{L^2 (\lambda)} \lesssim \| S f \|\ci{L^2 (\lambda)}$ by
  using that $S^2 = - \tmop{Id}$. Therefore
  \[ \left\| S \left( \sum_{K \in \mathcal{D} (I)} b\ci{K} h\ci{K} \right)
     \right\|^2\ci{L^2 (\hat{I}, \lambda)} \gtrsim \left\| \sum_{K \in
     \mathcal{D} (I)} b\ci{K} h\ci{K} \right\|\ci{L^2 (\hat{I}, \lambda)}^2 \gtrsim
     \int_I | b (x) - \langle b \rangle\ci{I} |^2 \lambda (x) \mathd x. \]
  
  Note that
  \[\int_I | b (x) - \langle b \rangle\ci{I} |^2 \lambda (x) \mathd x =
     \frac{1}{{\mu} (I)} \int_I | b (x) - \langle b \rangle\ci{I} |^2 \lambda
     (x) \mathd x \cdot {\mu} (I),\]
  so these considerations tell us that the lower BMO estimate is seen when
  testing on $\1\ci{I}$ and taking the supremum in $I.$
\end{proof}

The same considerations hold true for the tensor commutator $\mathcal{C}_b$ using the pointwise identities in section  \ref{littlebmo}.

Notice also, that the same reasoning, following the above calculation in section \ref{rect} provides a weighted lower bound for the iterated commutator in terms of rectangular Bloom BMO defined by $$\|b\|^2\ci{\tmop{BMO}(\mu,\lambda)\ti{rect}} =\sup_R\frac1{\mu(R)}\int_R\left|\sum_{K\in \bfD(R)}b\ci K h\ci K\right|^2\lambda,$$ where the supremum runs over dyadic rectangles $R$.

\subsection{Weighted estimate in \texorpdfstring{$L^p$}{Lp} using the kernel}

\begin{thm}
  There holds for dyadic bi-parameter $A_p$ weights $\mu,\lambda$, $1<p<\infty$
  \[ \| b \|\ci{p,\tmop{bmo} (\mu,\lambda)} \lesssim \| C_b \|\ci{L^p ({\mu})
     \rightarrow L^p (\lambda)} \lesssim \| b \|\ci{p,\tmop{bmo} (\mu,\lambda)} \]
  with constants independent of the symbol. 
  \end{thm}

\subsubsection{A heuristic argument} Let $1<p<\infty$, and let $\mu,\lambda$ be any dyadic bi-parameter $A_p$ weights on $\R^2$. We want to prove that
\begin{equation}
\label{weighted goal}
\left(\frac{1}{\mu(I\times J)}\int_{I\times J}|b(x)-\La b\Ra\ci{I\times J}|^{p}\lambda(x)\mathd x\right)^{1/p}\lesssim\Vert[S_1S_2,b]\Vert\ci{L^p(\mu)\rightarrow L^p(\lambda)},
\end{equation}
for all $I,J\in\cD$.

Let $I,J\in\cD$ be arbitrary. Identically to the unweighted case we formally have
\begin{align*}
&|I\times J|\1\ci{I\times J}(x)(b(x)-\La b\Ra\ci{I\times J})\\
&=\sum_{\substack{K\in\cD(I)\\L\in\cD(J)}}\sum_{\e,\delta\in\lbrace-,+\rbrace}\e\delta[S_1S_2,b]\left(\frac{\1\ci{K_{\e}\times L_{\delta}}}{h\ci{K_{\e}\times L_{\delta}}}\right)(x)\cdot\frac{\1\ci{K_{-\e}\times L_{-\delta}}(x)}{h\ci{K_{-\e}\times L_{-\delta}}(x)}.
\end{align*}
Thus, taking the $L^p_{x}(\lambda)$ norms of both handsides and then using the triangle inequality we deduce
\begin{align*}
&\Vert |I\times J|\1\ci{I\times J}(x)(b(x)-\La b\Ra\ci{I\times J})\Vert\ci{L^p_x(\lambda)}\\
&\leq
\Vert [S_1S_2,b]\Vert\ci{L^p(\mu)\rightarrow L^p(\lambda)}\sum_{\substack{K\in\cD(I)\setminus\lbrace I\rbrace\\L\in\cD(J)\setminus\lbrace J\rbrace}}\left\Vert \frac{\1\ci{K\times L}(x)}{h\ci{K\times L}(x)}\right\Vert\ci{L^{\infty}_{x}(\R^2)}\left\Vert\frac{\1\ci{K'\times L'}}{h\ci{K'\times L'}}\right\Vert\ci{L^p(\mu)},
\end{align*}
where we recall that $K'$ denotes the sibling of a dyadic interval $K$. We have
\begin{equation*}
\left\Vert \frac{\1\ci{K\times L}(x)}{h\ci{K\times L}(x)}\right\Vert\ci{L^{\infty}_{x}(\R^2)}=\sqrt{|K\times L|},
\end{equation*}
\begin{equation*}
\left\Vert\frac{\1\ci{K'\times L'}}{h\ci{K'\times L'}}\right\Vert\ci{L^p(\mu)}=|K\times L|^{\frac{1}{2}}(\mu(K'\times L'))^{\frac{1}{p}}.
\end{equation*}
Thus
\begin{align*}
&\sum_{\substack{K\in\cD(I)\setminus\lbrace I\rbrace\\L\in\cD(J)\setminus\lbrace J\rbrace}}\left\Vert \frac{\1\ci{K\times L}(x)}{h\ci{K\times L}(x)}\right\Vert\ci{L^{\infty}_{x}(\R^2)}\left\Vert\frac{\1\ci{K'\times L'}}{h\ci{K'\times L'}}\right\Vert\ci{L^p(\mu)}=\sum_{\substack{K\in\cD(I)\setminus\lbrace I\rbrace\\L\in\cD(J)\setminus\lbrace J\rbrace}}|K\times L|(\mu(K'\times L'))^{\frac{1}{p}}.
\end{align*}
We have
\begin{align*}
&\sum_{\substack{K\in\cD(I)\setminus\lbrace I\rbrace\\L\in\cD(J)\setminus\lbrace J\rbrace}}|K\times L|(\mu(K'\times L'))^{\frac{1}{p}}=\sum_{n,m=1}^{\infty}\sum_{\substack{K\in\text{ch}_{n}(I)\\L\in\text{ch}_{m}(J)}}2^{-n-m}|I\times J|(\mu(K'\times L'))^{1/p}\\
&\leq\sum_{n,m=1}^{\infty}2^{-n-m}|I\times J|\bigg(\sum_{\substack{K\in\text{ch}_{n}(I)\\L\in\text{ch}_{m}(J)}}1\bigg)^{1/p'}\bigg(\sum_{\substack{K\in\text{ch}_{n}(I)\\L\in\text{ch}_{m}(J)}}\mu(K'\times L')\bigg)^{1/p}\\
&=c_p^2|I\times J|(\mu(I\times J))^{1/p}.
\end{align*}
Thus
\begin{equation*}
|I\times J|\left(\int_{I\times J}|b(x)-\La b\Ra\ci{I\times J}|^p\lambda(x)\mathd x\right)^{1/p}\leq\Vert [S_1S_2,b]\Vert\ci{L^p(\mu)\rightarrow L^p(\lambda)}c_p^2|I\times J|(\mu(I\times J))^{1/p},
\end{equation*}
which immediately implies \eqref{weighted goal} (with constant $c_p^2$).

\subsubsection{A rigorous argument}

Let $I,J\in\cD$ be arbitrary. Identically to the unweighted case we have for \emph{almost every} $x\in\R^2$
\begin{align*}
&|I\times J|\1\ci{I\times J}(x)(b(x)-\La b\Ra\ci{I\times J})\\
&=\lim_{n\rightarrow\infty}\bigg[\sum_{\substack{K\in\cD(I)\\L\in\cD(J)\\2^{-n}\leq|K|,|L|\leq 2^{n}}}\sum_{\e,\delta\in\lbrace-,+\rbrace}\e\delta[S_1S_2,b]\left(\frac{\1\ci{K_{\e}\times L_{\delta}}}{h\ci{K_{\e}\times L_{\delta}}}\right)(x)\cdot\frac{\1\ci{K_{-\e}\times L_{-\delta}}(x)}{h\ci{K_{-\e}\times L_{-\delta}}(x)}\bigg],
\end{align*}
Therefore, taking the $L^p_{x}(\lambda)$ norm of both handsides, using Fatou's lemma (to pass the limit outside the $L^p_{x}(\lambda)$ norm)  and then working exactly as in the heuristic weighted argument, we conclude again \eqref{weighted goal} (with the same constant $c_p^2$ as in the heuristic argument).

\section{Lower estimates via the kernel for more general shifts}

In this section we show that the variant of the classical argument due to Coifman--Rochberg--Weiss \cite{CRW} that was used above to obtain lower bounds for commutators involving the shift $S$ can also handle slightly more general shifts, including the one considered in \cite{HTV}.

\subsection{Setup} Let $i,j$ be nonnegative integers. Let $T$ be a Haar shift on the real line of complexity $(i,j)$, i.e. the action of $T$ on (suitable) functions $f$ on the real line is given by
\begin{equation}
\label{shift}
Tf=\sum_{I\in\cD}\sum_{\substack{K\in\ch_{i}(I)\\L\in\ch_{j}(I)}}c\ci{KL}^{I}f\ci{K}h\ci{L},
\end{equation}
where the \emph{shift coefficients} $c\ci{KL}^{I}$ are complex numbers satisfying the bound
\begin{equation*}
|c\ci{KL}^{I}|\leq 2^{-(i+j)/2},
\end{equation*}
for all $K\in\ch_{i}(I), L\in\ch_{j}(I)$ and for all $I\in\cD$. This definition of Haar shifts follows \cite[p. 34]{Per}. The operator $T$ has a formal kernel. Namely, one can write
\begin{equation*}
Tf(x)=\int_{\R}\cK(x,y)f(y)\mathd y,
\end{equation*}
where
\begin{equation}
\label{general kernel}
\cK(x,y):=\sum_{I\in\cD}\sum_{\substack{K\in\ch_{i}(I)\\L\in\ch_{j}(I)}}c\ci{KL}^{I}h\ci{K}(y)h\ci{L}(x),
\end{equation}
for all $x,y\in\R$ with $x\neq y$, and $\cK(x,x):=0$. While the kernel representation for the shift $T$ is formal, the kernel is well-defined pointwise. Indeed, let $x,y\in\R$ with $x\neq y$ be arbitrary. If there is no dyadic interval containing both $x$ and $y$, then $\cK(x,y)=0$. Now assume that there is a dyadic interval containing both $x$ and $y$, and let $J$ be the minimal such dyadic interval. Let $(J_{n})^{\infty}_{n=1}$ be the strictly increasing sequence of all dyadic ancestors of $J$, so that
\begin{equation*}
J=:J_0\subsetneq J_1\subsetneq J_2\subsetneq J_3\subsetneq\ldots.
\end{equation*}
Then, we have
\begin{align*}
&\sum_{I\in\cD}\sum_{\substack{K\in\ch_{i}(I)\\L\in\ch_{j}(I)}}|c\ci{KL}^{I}h\ci{K}(y)h\ci{L}(x)|=
\sum_{n=0}^{\infty}\sum_{\substack{K\in\ch_{i}(J_n)\\L\in\ch_{j}(J_n)}}|c\ci{KL}^{J_n}h\ci{K}(y)h\ci{L}(x)|\\&
\leq2^{-(i+j)/2}\cdot2^{(i+j)/2}\cdot\sum_{n=0}^{\infty}\frac{1}{|J_n|}=\frac{2}{|J|}.
\end{align*}
Since the Haar function over an interval is constant in the dyadic children of that interval, the above computation shows that we can write
\begin{equation}
\label{reduced kernel expression}
\cK(x,y)=\sum_{I\in\cD}\sum_{\substack{K\in\ch_{i+1}(I)\\L\in\ch_{j+1}(I)}}a\ci{KL}^{I}\1\ci{K}(y)\1\ci{L}(x),\qquad\forall x,y\in\R,~x\neq y,
\end{equation}
where the \emph{coefficients} $a\ci{KL}^{I}$ satisfy
\begin{equation*}
|a\ci{KL}^{I}|\leq\frac{2}{|I|},
\end{equation*}
and also $a\ci{KL}^{I}=0$ whenever there is $J\in\text{ch}(I)$ with $K\cup L\subseteq J$.
Note that the only dyadic interval in the first sum for which we have non-zero terms is the minimal dyadic interval containing both $x$ and $y$ (if there is one), so that for each fixed $x,y\in\R$ with $x\neq y$, there is at most one nonzero term in the sum in \eqref{reduced kernel expression}. Namely, if $I$ is the minimal dyadic interval containing both $x$ and $y$, $K$ is the unique interval in $\ch_{i+1}(I)$ containing $y$, and $L$ is the unique interval in $\ch_{j+1}(I)$ containing $x$, then $\cK(x,y)=a\ci{KL}^{I}$.

\subsection{A non-degeneracy condition}

Now we assume the following non-degeneracy condition: there exists some constant $c>0$ (depending only on $i,j$), such that whenever there is no $J\in\ch(I)$ such that $K\cup L\subseteq J$, then there holds
\begin{equation}
\label{non-degeneracy condition}
|a\ci{KL}^{I}|\geq\frac{1}{c|I|}.
\end{equation}
Then, in particular we deduce that if $x\neq y$ and there exists some dyadic interval containing both $x$ and $y$, then $\cK(x,y)\neq0$, in fact
\begin{equation*}
\frac{1}{c|I|}\leq|\cK(x,y)|\leq\frac{2}{|I|},
\end{equation*}
where $I$ is the minimal dyadic interval containing both $x$ and $y$. In Section \ref{sec:examples} below we give representative examples of classes of Haar shifts for which condition \eqref{non-degeneracy condition} is satisfied.

\subsection{The inverse kernel}

For all $x\in\R$, set
\begin{equation*}
A_{x}:=\lbrace y\in\R\setminus\lbrace x\rbrace:~ \cK(x,y)\neq0\rbrace.
\end{equation*}
Note that for all $x\in\R$ and for all $J\in\cD$, we have $J\setminus\lbrace x\rbrace\subseteq A_{x}$. Then, we can write
\begin{equation*}
\frac{\1\ci{A_{x}}(y)}{\cK(x,y)}=\sum_{I\in\cD}\sum_{\substack{K\in\ch_{i+1}(I)\\L\in\ch_{j+1}(I)}}b\ci{KL}^{I}\1\ci{K}(y)\1\ci{L}(x),\qquad\forall x,y\in\R,~x\neq y,
\end{equation*}
where
\begin{equation*}
|b\ci{KL}^{I}|\leq c|I|,
\end{equation*}
and also $b\ci{KL}^{I}=0$ whenever there is $J\in\text{ch}(I)$ with $K\cup L\subseteq J$.
Moreover, as before, this means that the only relevant dyadic interval in the first sum is the minimal dyadic interval $I$ containing both $x$ and $y$ (if there is one). In particular, for all $J\in\cD$, we have the localized version
\begin{equation*}
\frac{\1\ci{J}(x)\cdot\1\ci{J}(y)\cdot\1\ci{A_{x}}(y)}{\cK(x,y)}=\sum_{I\in\cD(J)}\sum_{\substack{K\in\ch_{i+1}(I)\\L\in\ch_{j+1}(I)}}b\ci{KL}^{I}\1\ci{K}(y)\1\ci{L}(x),\qquad\forall x,y\in\R,~x\neq y.
\end{equation*}
Note that we have adopted everywhere the convention $\frac{0}{0}=0$.

\subsection{Lower BMO bounds}

Let $1<p<\infty$, and let $\mu,\lambda$ be dyadic $A_p$ weights on $\R$. Let $J\in\cD$. We will show that
\begin{equation*}
\left(\frac{1}{\mu(J)}\int_{J}|b(x)-\La b\Ra\ci{I}|^{p}\lambda(x)\mathd x\right)^{1/p}\leq cc_{p}2^{\frac{i+1}{p'}}\Vert[T,b]\Vert\ci{L^{p}(\mu)\rightarrow L^{p}(\lambda)}.
\end{equation*}
Let us first give a heuristic argument. We have
\begin{align*}
&|J|\1\ci{J}(x)(b(x)-\La b\Ra\ci{J})=\1\ci{J}(x)\int_{J}(b(x)-b(y))\mathd y\\
&=\1\ci{J}(x)\int_{A_x}(b(x)-b(y))\1\ci{J}(y)\mathd y\\
&=\1\ci{J}(x)\int_{\R}(b(x)-b(y))\1\ci{J}(y)\cK(x,y)\frac{\1\ci{A_x}(y)}{\cK(x,y)}\mathd y\\
&=[T,b]\left(\1\ci{J}(x)\cdot\1\ci{J}(y)\cdot\frac{\1\ci{A_x}(y)}{\cK(x,\fdot)}\right)(x)\\
&=\sum_{I\in\cD(J)}\sum_{\substack{K\in\ch_{i+1}(I)\\L\in\ch_{j+1}(I)}}b\ci{KL}^{I}[T,b](\1\ci{K})(x)\1\ci{L}(x).
\end{align*}
Taking absolute values and using the triangle inequality as well as the non-degeneracy assumption for the coefficients $b\ci{KL}^{I}$ we get
\begin{align*}
|J|\1\ci{J}(x)|b(x)-\La b\Ra\ci{J}|&\leq c\sum_{I\in\cD(J)}|I|\sum_{\substack{K\in\ch_{i+1}(I)\\L\in\ch_{j+1}(I)}}|[T,b](\1\ci{K})(x)|\1\ci{L}(x)\\
&=c\sum_{I\in\cD(J)}|I|\left(\sum_{K\in\ch_{i+1}(I)}|[T,b](\1\ci{K})(x)|\right)\left(\sum_{L\in\ch_{j+1}(I)}\1\ci{L}(x)\right)\\
&=c\sum_{I\in\cD(J)}|I|\sum_{K\in\ch_{i+1}(I)}|[T,b](\1\ci{K})(x)|\1\ci{I}(x).
\end{align*}
Thus, taking $L^p(\lambda)$ norms and using the triangle inequality (for $L^p(\lambda)$ norms), we get
\begin{align*}
|J|^p\left(\int_{J}|b(x)-\La b\Ra\ci{J}|^{p}\lambda(x)\mathd x\right)^{1/p}&\leq c\sum_{I\in\cD(J)}|I|\sum_{K\in\ch_{i+1}(I)}\Vert [T,b](\1\ci{K})\1\ci{I}\Vert\ci{L^{p}(\lambda)}\\
&\leq c\sum_{I\in\cD(J)}|I|\sum_{K\in\ch_{i+1}(I)}\Vert [T,b]\Vert\ci{L^{p}(\mu)\rightarrow L^{p}(\lambda)}(\mu(K))^{1/p}.
\end{align*}
We have
\begin{align*}
\sum_{I\in\cD(J)}|I|\sum_{K\in\ch_{i+1}(I)}(\mu(K))^{1/p}&\leq\sum_{I\in\cD(J)}|I|\left(\sum_{K\in\ch_{i+1}(I)}\mu(K)\right)^{1/p}\left(\sum_{K\in\ch_{i+1}(I)}1\right)^{1/p'}\\
&=2^{\frac{i+1}{p'}}\sum_{I\in\cD(J)}|I|(\mu(I))^{1/p}=2^{\frac{i+1}{p'}}\sum_{n=0}^{\infty}2^{-n}|J|\sum_{I\in\ch_{n}(J)}(\mu(I))^{1/p}\\
&\leq2^{\frac{i+1}{p'}}\sum_{n=0}^{\infty}2^{-n}|J|\left(\sum_{I\in\ch_{n}(J)}\mu(I)\right)^{1/p}\left(\sum_{I\in\ch_{n}(J)}1\right)^{1/p'}\\
&=2^{\frac{i+1}{p'}}|J|(\mu(J))^{1/p}\sum_{n=0}^{\infty}2^{-n/p}=c_{p}2^{\frac{i+1}{p'}}|J|(\mu(J))^{1/p}.
\end{align*}
The claim follows.

To make the above argument formal, one has just to truncate the kernel, similarly to the case of the shift $S$ treated above, with the only difference that in the present more general case one has to truncate \eqref{reduced kernel expression}, instead of \eqref{general kernel}.

A bi-parameter variant of the previous argument also gives little BMO lower bounds for the commutator $[T\otimes T,b]$ (under the same non-degeneracy condition as in the one-parameter case), similarly to the case of the commutator $[S\otimes S,b]$ treated above. In fact, a bi-parameter variant of the above argument also yields little BMO lower bounds for commutators of the form $[\mathbb{S},b],$ where $\mathbb{S}$ is an arbitrary biparameter shift $\mathbb{S}$ (not necessarily of tensor type) satisfying a bi-parameter non-degeneracy condition analogous to \eqref{non-degeneracy condition}.

\subsection{Cases where the non-degeneracy condition \texorpdfstring{\eqref{non-degeneracy condition}}{nondegcondition} is satisfied}
\label{sec:examples}

We give here two representative special cases in which the non-degeneracy condition \eqref{non-degeneracy condition} is satisfied.

\subsubsection{Purely mixing shifts of complexity \texorpdfstring{$(i,i)$}{(i,i)} with mildly varying coefficients}
Take the shift $T$ to have complexity $(i,i),$ so that it has the form
\begin{equation*}
Tf=2^{-i}\sum_{I\in\cD}\sum_{\substack{K\in\ch_{i}(I)\\ L\in\ch_{i}(I)}}c^{I}\ci{KL}f\ci{K}h\ci{L}.
\end{equation*}
Moreover, assume as well that this shift is ``purely mixing'' in the sense that for any $I \in \cD$ and any $K \in \ch_i(I)$ it holds that $c^{I}\ci{KK} = 0,$ and that all other coefficients vary mildly in the sense that there exists $b \in [1,2^i/(2^i-1))$ such that
\begin{equation*}
  1 \leq |c^{I}\ci{KL}| \leq b
\end{equation*}
for all $I\in\cD$ and all $K,L \in \ch_{i}(I)$ with $K \neq L.$
In this case, the kernel for this shift is of the form
\begin{equation*}
\cK(x,y)=2^{-i}\sum_{I\in\cD}\sum_{\substack{K\in\ch_{i}(I)\\ L\in\ch_{i}(I)}}c^{I}\ci{KL}h\ci{K}(y)h\ci{L}(x),
\end{equation*}
with the previous conditions on the coefficients $c^{I}\ci{KL}$ of $Tf.$

To check the non-degeneracy condition \eqref{non-degeneracy condition} for this shift, we estimate directly the coefficients $a^{I}\ci{KL}$ of the kernel $\cK$ in \eqref{reduced kernel expression}.
Observe that for $I\in\cD$ and $K,L \in \ch_{i+1}(I)$ such that $K \cup L$ is not contained in a dyadic child of $I,$ and for $x \in L$ and $y \in K,$ we have that $\cK(x,y) = a^{I}\ci{KL}.$
Thus, fix $I,K,L$ dyadic intervals as before and $x \in L$ and $y \in K,$ and take $(I_n)^{\infty}_{n=1}$ to be the increasing sequence of dyadic ancestors of $I_0 := I.$
Since $c^{J}\ci{PP} = 0$ for any $J \in \cD$ and $P \in \ch_i(J),$ we have that
\begin{equation*}
\cK(x,y)=2^{-i}\sum_{n=0}^{i-1}\sum_{\substack{K\in\ch_{i}(I_n)\\ L\in\ch_{i}(I_n)}}c^{I_n}\ci{KL}h\ci{K}(y)h\ci{L}(x).
\end{equation*}
For this particular $x$ and $y$ we can choose a sequence of signs $(\e_{n})^{i-1}_{n=0}$ in $\lbrace -1, 1\rbrace$ such that
\begin{equation*}
\cK(x,y)=2^{-i}\sum_{n=0}^{i-1} \frac{\e_n c_n}{2^{-i}|I_n|} = \sum_{n=0}^{i-1} \frac{\e_n c_n}{2^n|I|},
\end{equation*}
where $c_n$ are complex numbers such that
\begin{equation*}
  1 \leq |c_n| \leq b
\end{equation*}
for $n = 0, \ldots, i-1.$
Observe that
\begin{equation*}
  \left| \sum_{n=1}^{i-1} \frac{\e_n c_n}{2^n|I|} \right| \leq \frac{b}{|I|} \sum_{n=1}^{i-1} 2^{-n} = \frac{(2^{i-1}-1)b}{2^{i-1}|I|}
\end{equation*}
(here we take the sum to be just $0$ if $i = 1$).
Thus, by the triangle inequality, we get that
\begin{equation*}
  |\cK(x,y)| = \left| \frac{\e_0 c_0}{|I|} + \sum_{n=1}^{i-1} \frac{\e_n c_n}{2^n|I|} \right|
  \geq \left( \frac{1}{|I|} - \frac{(2^{i-1}-1)b}{2^{i-1}|I|} \right) = \left( 1 - \frac{(2^{i-1}-1)b}{2^{i-1}} \right) \frac{1}{|I|}
\end{equation*}
and the condition on $b$ yields the non-degeneracy condition \eqref{non-degeneracy condition}.

\subsubsection{``Sliced'' shifts with mildly varying coefficients} Let $\cD\ti{e}$ be the family of all \text{even} dyadic intervals, i.e.
\begin{equation*}
\cD\ti{e}:=\lbrace I\in\cD:~\log_{2}(|I|)\text{ is even}\rbrace.
\end{equation*}
Assume that the shift $T$ is ``sliced" and that its coefficients do not vary too much, in the sense that
\begin{equation*}
Tf=2^{-(i+j)/2}\sum_{I\in\cD\ti{e}}\sum_{\substack{K\in\ch_{i}(I)\\ L\in\ch_{j}(I)}}c^{I}\ci{KL}f\ci{K}h\ci{L},
\end{equation*}
where there exists $b\in[1,3)$ such that
\begin{equation*}
1\leq|c\ci{KL}^{I}|\leq b,
\end{equation*}
for all $K\in\ch_{i}(I)$, $L\in\ch_{j}(I)$ and for all $I\in\cD_{e}$. In particular, $c^{I}\ci{KL}=0$, for all $K\in\ch_{i}(I)$, $L\in\ch_{j}(I)$ and for all $I\in\cD\setminus\cD_{e}$. Note that
\begin{equation*}
\cK(x,y)=2^{-(i+j)/2}\sum_{I\in\cD\ti{e}}\sum_{\substack{K\in\ch_{i}(I)\\ L\in\ch_{j}(I)}}c^{I}\ci{KL}h\ci{K}(y)h\ci{L}(x).
\end{equation*}

We estimate directly the coefficients $a\ci{KL}^{I}$ in \eqref{reduced kernel expression}.

Let $I\in\cD$, and $K\in\ch_{i+1}(I)$, $L\in\ch_{j+1}(I)$ such that there is no dyadic child of $I$ containing both $K$ and $L$. Pick $x\in L$ and $y\in K$. Then, it is clear that $a^{I}\ci{KL}=\cK(x,y)$.

Let us estimate $\cK(x,y)$. Let $(I_n)^{\infty}_{n=1}$ be the strictly increasing sequence of all dyadic ancestors of $I_{0}:=I$. Then, it is clear that there exist a sequence $(\e_{n})^{\infty}_{n=0}$ in $\lbrace-1,0,1\rbrace$ and a sequence $(c_{n})^{\infty}_{n=0}$ of complex numbers such that
\begin{equation*}
\cK(x,y)=\sum_{n=0}^{\infty}\frac{\e_{n}c_n}{|I_n|},
\end{equation*}
and
\begin{equation*}
1\leq|c_{n}|\leq b,\qquad\forall n=0,1,2\ldots,
\end{equation*}
and $\e_{n}\in\lbrace-1,1\rbrace$ if $I_{n}\in\cD_{e}$, while $\e_{n}=0$ if $I_{n}\in\cD\setminus\cD_{e}$, for all $n=0,1,2\ldots$. We now distinguish two cases.

\textbf{Case 1.} Assume that $I\in\cD_{e}$. Then
\begin{equation*}
\cK(x,y)=\frac{\e_{0}c_{0}}{|I|}+\sum_{n=1}^{\infty}\frac{\e_{2n}c_{2n}}{|I_{2n}|}.
\end{equation*}
We notice that
\begin{align*}
\left|\sum_{n=1}^{\infty}\frac{\e_{2n}c_{2n}}{|I_{2n}|}\right|\leq\sum_{n=1}^{\infty}\frac{b}{|I_{2n}|}=\sum_{n=1}^{\infty}\frac{b}{2^{2n}|I|}=\frac{b}{3|I|},
\end{align*}
therefore by the triangle inequality we deduce
\begin{equation*}
|\cK(x,y)|\geq\left(1-\frac{b}{3}\right)\frac{1}{|I|}.
\end{equation*}

\textbf{Case 2.} Assume that $J\in\cD\setminus\cD_{e}$. Then
\begin{equation*}
\cK(x,y)=\frac{\e_{1}c_{1}}{|I_1|}+\sum_{n=1}^{\infty}\frac{\e_{2n+1}c_{2n+1}}{|I_{2n+1}|}.
\end{equation*}
We notice that
\begin{align*}
\left|\sum_{n=1}^{\infty}\frac{\e_{2n+1}c_{2n+1}}{|I_{2n+1}|}\right|\leq\sum_{n=1}^{\infty}\frac{b}{|I_{2n+1}|}=\sum_{n=1}^{\infty}\frac{b}{2^{2n}|I_1|}=\frac{b}{3|I_1|},
\end{align*}
therefore by the triangle inequality we deduce
\begin{equation*}
|\cK(x,y)|\geq\left(1-\frac{b}{3}\right)\frac{1}{|I_1|}=\frac{1}{2}\left(1-\frac{b}{3}\right)\frac{1}{|I|},
\end{equation*}
concluding the proof.

\subsection{A question} The non-degeneracy condition \eqref{non-degeneracy condition} might be considered a bit too strong from a more general point of view, and especially from the point of view of \cz operators. A far weaker and perhaps more natural non-degeneracy condition is the following. There exists some $c>0$ (depending on $i$ and $j$), such that for all $I\in\cD$, for all $K\in\ch_{i+1}(I)$, there exists \emph{some} $L\in\ch_{j+1}(I)$ (depending on $K$), such that
\begin{equation*}
|a^{I}\ci{KL}|\geq\frac{1}{c|I|}.
\end{equation*}
This is equivalent to saying that for all $y\in\R$ and for all dyadic intervals $I$ containing $y$, there exists $x\in I$, such that $I$ is the minimal dyadic interval containing both $x$ and $y$ and $|\cK(x,y)|\geq\frac{1}{c|I|}$. This is the direct dyadic analog of the non-degeneracy condition for \cz operators considered by Hyt\ddoto nen in \cite{Hyt}, where it is shown that it is sufficient for some lower BMO bounds for commutators with \cz operators. Note that the proofs of lower BMO bounds in \cite{Hyt} depend heavily on variants of weak factorization. It is not immediately clear to us whether our methods can be adapted to cover shifts that satisfy only this much weaker non-degeneracy condition.


\begin{thebibliography}{10}

\bibitem[B]{B}
{Alain Bernard},
{\it Espaces {$H^{1}$} de martingales {\`a} deux indices. {D}ualit\'{e}
  avec les martingales de type BMO},
Bull. Sci. Math. (2),
volume  103,
number 3, 
1979, 
pages 297-303.


\bibitem[BP]{BP}

 {\'{O}scar Blasco, Sandra Pott},
  {\it Dyadic BMO on the bidisk},
  Rev. Mat. Iberoamericana,
  volume 21,
  number 2,
  2005,
  pages 483-510.
  
  \bibitem[Bl]{Bl} 
  {Steven Bloom}, 
  {\it A commutator theorem and weighted BMO}, 
  Trans. Amer. Math. Soc.,
  volume 292,
  number 1,
  1985,
  pages 103–122. 

 
 \bibitem[CRW]{CRW} 
 { Ronald Coifman, Richard Rochberg, Guido Weiss},
     {\it Factorization theorems for Hardy spaces in several variables},
   Ann. of Math. (2),
    volume 103,
    number 3,
      1976,
     pages 611-635.
     
     
 \bibitem[FS]{FS}
{Sarah Ferguson, Cora Sadosky},
{\it Characterizations of bounded mean oscillation on the polydisk in terms of Hankel operators and Carleson measures},
   J. Anal. Math.,
    volume 81,
      2000,
     pages 239-267.
    

     
\bibitem[HLW]{HLW}
{Irina Holmes, Michael Lacey, Brett Wick},
{\it Bloom's Inequality: Commutators in a Two-Weight Setting},
Arch. Math. (Basel),
volume 106,
number 1,
2016,
pages 53-63.

\bibitem[HLW2]{HLW2}
{Irina Holmes, Michael Lacey, Brett Wick}, 
{\it Commutators in the two-weight setting}, 
Math. Ann.,
volume 367,
number 1-2,
2017,
pages 51–80.


\bibitem[HPW]{HPW}
{Irina Holmes, Stefanie Petermichl, Brett Wick},
{\it Weighted little bmo and two-weight inequalities for Journ\'{e} commutators},
Anal. PDE,
volume 11,
number 7,
2018,
pages 1693-1740.
 
\bibitem[HTV]{HTV}
{Irina Holmes, Sergei Treil, Alexander Volberg},
{\it Dyadic bi-parameter simple commutator and dyadic little BMO},
arXiv:2012.05376, 
2020,
pages 1-14.


\bibitem[Hyt]{Hyt} Tuomas Hyt\ddoto nen, \emph{The $L^{p}$-to-$L^{q}$ boundedness of commutators with applications to the Jacobian operator}, 
arXiv:1804.11167,
2018,
pages 1-35.


\bibitem[KS]{KS}
{Spyridon Kakaroumpas, Odí Soler i Gibert},
{\it Dyadic product BMO in the Bloom setting},
arXiv:2011.01769,
2020,
pages 1-31.
 
\bibitem[N]{N}
{Zeev Nehari},
{\it On bounded bilinear forms},
Ann. of Math. (2),
volume 65,
1957,
pages 153-162.
     
     
\bibitem[OP]{OP}
{Yumeng Ou, Stefanie Petermichl},
{\it Little BMO and Journ\'{e} commutators},
Theta Ser. Adv. Math.,
volume 19,
2017,
pages 207-219.


\bibitem[Per]{Per} Maria Cristina Pereyra, \emph{Dyadic Harmonic Analysis and Weighted Inequalities: The Sparse Revolution}, In: Aldroubi A., Cabrelli C., Jaffard S., Molter U. (eds) New Trends in Applied Harmonic Analysis, Volume 2, Applied and Numerical Harmonic Analysis, Birkh\ddota user, Cham, vol.~2, p.~159--239 (2019)


\bibitem[P]{P}
{Stefanie Petermichl},
{\it Dyadic shifts and a logarithmic estimate for Hankel operators with matrix symbol},
Comptes Rendus Acad. Sci. Paris,
volume 1,
number 1,
2000,
pages 455-460.





\end{thebibliography}
\end{document}